\documentclass[12pt,a4paper]{article}

\usepackage[OT1]{fontenc}
\usepackage[utf8]{inputenc}

\usepackage{mathrsfs,units}
\usepackage{hyperref}

\usepackage[english]{babel}
\usepackage{amsfonts}
\usepackage{amsthm}
\usepackage{graphicx}
\usepackage{mathtools}  % (location: texlive-latex3)

\newtheorem*{theorem*}{Theorem}
\newtheorem{teo}{Theorem}[section]
\newtheorem{prop}[teo]{Proposition}
\newtheorem{lem}[teo]{Lemma}
\newtheorem{coro}[teo]{Corollary}

\newtheorem{rem}[teo]{Remark}

\def\bh{{\cal B}({\cal H})}
\def\kh{{\mathscr K}({\cal H})}
\def\ka{{\mathscr K}({\widetilde{\a}})}
\def\pa{\mathcal P(\a)}
\def\elep{{\mathscr L}^p({\a})}
\def\l1{{\mathscr L}^1({\a})}
\def\loclp{{\mathscr L}^p_{loc}({\a})}
\def\lo1{{\mathscr L}^1_{loc}({\a})}
\def\loc2{{\mathscr L}^2_{loc}({\a})}

\def\ran{{\rm{Ran\, }}}

\def\a{\mathcal A}
\def\be{b_{\varepsilon}}

\def\pp1{\overline{p_1}}
\def\ppp1{\overline{\overline{p_1}}}

\begin{document}

\title{\vspace*{0cm}The case of equality in Young's inequality for the $s$-numbers in semi-finite von Neumann algebras.}

% 15A42 Linear and multilinear algebra; matrix theory --  Inequalities involving eigenvalues and eigenvectors
% 15A45 Linear and multilinear algebra; matrix theory  -- Miscellaneous inequalities involving matrices

% 47A30  Operator Theory -- Norms (inequalities, more than one norm, etc.)

% 47A63 Operator Theory --  Operator inequalities

\date{}
\author{G. Larotonda\footnote{Supported by Instituto Argentino de Matem\'atica(CONICET), Universidad de Buenos Aires and ANPCyT.}}

\maketitle

\setlength{\parindent}{0cm} %% para que no indente los parrafos nuevos

\begin{abstract}
\noindent For a semi-finite von Neumann algebra $\a$, we study the case of equality in Young's inequality of $s$-numbers for a pair of $\tau$-measurable operators $a,b$, and we prove that equality is only possible if $|a|^p=|b|^q$. We also extend the result to unbounded operators affiliated with $\a$, and relate this problem with other symmetric norm Young inequalities.
\end{abstract}

\section{Introduction}\label{intro1}

The well-known inequality, valid for $p>1$ and $1/p+1/q=1$, named after W. H. Young, is usually stated as
$$
\alpha\beta\le \nicefrac{1}{p}\,\alpha^p+\nicefrac{1}{q}\,\beta^q
$$
for any $\alpha,\beta\in\mathbb R^+$, \textit{with equality if and only if} $\alpha^p=\beta^q$. 

\medskip

In this paper, we establish an analogue for \textit{the case of equality} in the setting of operators affiliated to semi-finite von Neumann algebras. For more references and further discussion on the subject of Young's inequality for matrices and operators, we refer the reader to \cite{hycomp} where the proof is given for the particular case of compact operators in $\mathcal B(H)$ -the discrete (or atomic measure) case- of this fact. In particular, we remark that it was the fundamental paper by T. Ando \cite{ando} which initiated the study of Young's inequality for the singular values of $n\times n$ matrices.

\medskip

The emphasis in this paper is in the measure theoretic approach to operators affiliated with a semi-finite von Neumann algebra, since the approach by induction used in \cite{hycomp} is not at hand. The inequality for $s$-numbers of operators $a,b$ affiliated with a semi-finite von Neuman algebra $\mathcal A$, is stated as
\begin{equation}\label{snum}
\mu_s(ab^*)\le \mu_s\left(\nicefrac1p \,|a|^p +\nicefrac1q\, |b|^q\right),\;s> 0
\end{equation}
and extended here to unbounded operators; we are interested in the case of equality. 

\smallskip

We remark that this result includes all  semi-finite von Neumann algebras $\mathcal A$, since by a standard tensor product technique \cite[p.286]{fk}, we can always embed $\mathcal A$ into the diffuse algebra $\mathcal A\otimes \mathscr{L}^{\infty}([0,1],dt)$ without altering the $s$-numbers.

\bigskip

This paper is organized as follows: Section 2 presents the general facts about $s$-numbers recalling the well-known and establishing some simple lemmas used later. Section 3 deals with some simplifications and reductions of the problem to deal with it in full generality. Section 4, after certain technical propositions, contains the main result of this paper, Theorem \ref{elteo}, that states that equality holds for all $s$-numbers in (\ref{snum}) if and only if $|a|^p=|b|^q$, or equivalently, if equality of norms 
$$
\|ab^*\|_E=\|\nicefrac1p \,|a|^p +\nicefrac1q \,|b|^q\|_E
$$
holds for some \textit{strictly increasing symmetric norm} $\|\cdot\|_E$ (definition given in Section 4.1, just before the main theorem).

\section{Singular numbers in von Neumann algebras}

In this paper  $\a$ stands for a finite or semi-finite von Neumann algebra with faithful normal trace $\tau$, which when convenient we will assume represented in a complex Hilbert space $\mathcal H$. The set of (self-adjoint) projections in $\a$ will be denoted by $\pa$.

We consider the topology of \textit{convergence in measure} in $\a$: a neighbourhood of $0$ is given by 
$$
V(\varepsilon,\delta)=\{x\in \a: \exists\,  p\in \pa \mbox{ s.t. }\tau(1-p)<\delta \mbox{ and }\|xp\|<\varepsilon\}.
$$ 
We will denote with $\widetilde{\a}$ the closure in measure of $	\a$, therefore $\widetilde{\a}$ is \textit{the ring of $\tau$-measurable operators affiliated with $\a$}. In the atomic case, convergence in measure reduces to the norm topology, therefore $\widetilde{\a}=\a$ in that case. 

\smallskip

For $0\le x\in \widetilde{\a}$ and $s>0$, we denote the $s$-th singular number of $x$ by $\mu_s(x)$:
$$
\mu_s(x)=\inf\{\|xp\|: p\in \pa\mbox{ with } \tau(1-p)\le s\}.
$$
With $\mu_s(a)$ we denote the $s$-numbers of $|a|$, that is $\mu_s(a):=\mu_s(|a|)$. 
We remark that  $\lim\limits_{s\to 0^+}\mu_s(x)=\|x\|$  including the posibility of $+\infty$ when $x$ is unbounded. The standard reference on the subject is the paper by Fack and Kosaki \cite{fk}. 

\medskip

We comment here on some useful characterizations (Proposition 3.1 in \cite{harada}, Proposition 2.2,  Lemma 2.5, Proposition 3.1  in \cite{fk}).
\begin{itemize}
\item The variational (min-max) characterization:
\begin{equation}\label{minmax}
\mu_s(x)=\inf_{\substack{p\in\pa\\ \tau(1-p)\le s}}\big[\sup_{\substack{\xi\in \ran(p) \\ \|\xi\|=1}}\langle x\xi,\xi\rangle \big]=\sup_{\substack{p\in\pa\\ \tau(p)\ge s}}\big[\inf_{\substack{\xi\in \ran(p) \\ \|\xi\|=1}}\|x\xi\| \big].
\end{equation}

\item The distribution characterization: if $B\subset \mathbb R_{\ge 0}$ is a Borelian set and we denote $p^x(B)=\chi_B(|x|)$ (the range projections of $|x|$), then
$$
\mu_t(a)=\min\{s\ge 0: \tau(p^x(s,+\infty))\le t\}.
$$
From the very definition of $\widetilde{\a}$, the number $\tau(p^x(s,+\infty))$ is eventually finite, and moreover $\tau(p^x(s,+\infty))\to 0$ when $s\to \infty$.

\item For $x\in\widetilde{\a}$, the following are equivalent:
\begin{enumerate}
\item $\tau(p^x(t,+\infty))<+\infty$ for all $t>0$.
\item $\lim\limits_{t\to\infty}\mu_t(x)\to 0$.
\item There exists a sequence of bounded operators $x_n\in \l1$ such that $ x_n\to x$ in the measure topology.
\end{enumerate}

\begin{rem}\label{compa}
With any of these three characterizations, we say that $x$ is $\tau$-compact; these operators form a complete bilateral ideal in $\widetilde{\a}$ that we will denote by $\ka$; note that a $\tau$-compact operator is not necessarily bounded. We will denote with $\ka^+$ the positive $(x\ge 0)$ $\tau$-compact operators.

In the atomic case (when $\mathcal A=\bh$), then we recover the ordinary compact operators $\kh$. If $\{\lambda_k(x)\}_{k\in\mathbb N_0}$ denotes the usual singular values of $x$ (i.e. the eigenvalues of $|x|$), and we arrange them in a right-continuos decreasing function which is constant on $[k,k+1)$, then we obtain the distribution function $\mu_s(x)$ as follows:
$$
\mu_s(x)=\sum\limits_{k\in \mathbb N_0} \lambda_k\, \chi_{[k,k+1)}(s)
$$
\end{rem}

\end{itemize}

\medskip

In this lemma we collect some other known facts on $s$-numbers that we will use later.
\begin{lem}\label{singv}
Let $x,y\in \a$, $a,b\in\widetilde{\a}$. Then for each $s>0$,
\begin{enumerate}
\item $\mu_s(xay)\le \|x\|\|y\|\mu_s(a)$, and if $a\le b$ then $\mu_s(a)\le \mu_s(b)$.
\item $\mu_s(|ab^*|)=\mu_s(||a||b||)$.
\item $\mu_{s+t}(a+b)\le \mu_s(a)+\mu_t(b)$, $s,t\ge 0$.
\item If $p\in \pa$ then $\mu_s(ap)=0$ for each $s\ge \tau(p)$.
\item $\tau(|a|)=\int_0^{\infty}\mu_s(a)ds$.
\item If $a,b\ge 0$, $a\in \ka$, $ab=0$ and $\mu_s(a+b)=\mu_s(a)$ for all $s>0$, then $b=0$.
\end{enumerate}
\end{lem}
\begin{proof}
The first assertion is a consequence of the min-max characterization of the $s$-numbers. To prove the second, note that if $b=\nu|b|$ is the polar decomposition of $b$, a straitghtforward computation using the functional calculus shows that
$$
|ab^*|=\nu||a||b||\nu^*\,\mbox{ and }\, ||a||b||=\nu|ab^*|\nu.
$$
Then by the first item we obtain $\mu_s(|ab^*|)=\mu_s(||a||b||)$. The proof of the third, fourth and fifth assertion is due to Fack and Kosaki and can be found in their original paper \cite[Lemmas 2.5, 2.6 and Proposition 2.7]{fk}. The final assertion seems evident, but  requires some proof though. For $t>0$, let $p^a[t,+\infty)=\chi_{[t,+\infty)}(a)$ be the spectral projections of $a$, and likewise for $b,a+b$. Then $\mu_s(a+b)=\mu_s(a)$ for all $s>0$ implies (since $a$ is $\tau$-compact and $ab=0$) that 
$$
\tau(p^a(t,+\infty))=\tau(p^{a+b}(t,+\infty))=\tau(p^a(t,+\infty))+\tau(p^b(t,+\infty))
$$
for all $t>0$ (cf. \cite[Corollary 2.9]{fk}). Therefore $\tau(p^b(t,+\infty))=0$ for all $t>0$,  implying $b=0$.
\end{proof}

\section{Diffuse algebras}\label{difusemi}

Recall that an algebra is \textit{diffuse} if it has no minimal projections. Following Fack and Kosaki \cite[p.286]{fk}, we can always embed $\mathcal A$ into the diffuse algebra $\mathcal A\otimes \mathscr{L}^{\infty}([0,1],dt)$ without altering the $s$-numbers. Then, the following \cite[Lemma 2.1]{fk} will be useful later:

\begin{rem}\label{intemu}
If $x\ge 0$ is $\tau$-measurable, then for each $t\ge 0$
$$
\sup\{\tau(xp): p\in\pa,\, \tau(p)\le t\}=\int_0^t \mu_s(x)ds.
$$
\end{rem}

\subsection{Complete flags}

If $0\le x\in \ka$ and $\a$ is a diffuse von Neumann algebra, there exists an increasing  assignment  $\mathbb R_{\ge 0}\ni t\mapsto e_t\in\pa$ ($e_s\le e_t$ for $s\le t$) such that $\tau(e_t)=t$ for all $t\ge 0$ and 
$$
x=\int_0^{\infty} \mu_s(x)de(s).
$$

\textit{Note the analogy with the atomic case, where $x=\sum_{k\in\mathbb N} \lambda_k(x)p_k$ with $p_k$ the projection to the $\lambda_k$ eigenspace of $x$, and we assume the eigenvalues are arranged in decrasing order.}

\medskip

Since $e_0=0$, we denote $e(s,t)=e_t-e_s$ for $s\le t\in \mathbb R_{\ge 0}$ and since $\mathcal A$ is diffuse, 
$$
e_t-e_s=e(s,t)=e[s,t)=e(s,t]=e[s,t].
$$
The spectral resolution $\{e_t\}_{t\ge 0}$ is called a \textit{complete flag} for $x$; for more details on this useful constructions in diffuse semi-finite algebras, we refer the reader to the papers \cite{am1,am2} by Argerami and Massey. In particular, for each $t>0$,
$$
\int_0^t\mu_s(x)ds=\tau(xe_t).
$$

\medskip

\subsection{Equality of singular numbers, $\tau$-compact operators}

Let $(\mathcal A,\tau)$ be a semi-finite von Neumann algebra with semi-finite trace ($\tau(1)=+\infty$ here). 

In \cite[Theorem 1]{fm} Farenick and Manjegani proved the remarkable Young's inequality for the $s$-numbers: if $p>1$, $1/p+1/q=1$,  and $a,b\in \a$, then
\begin{equation}\label{fm}
\mu_s(ab^*)\le \mu_s\left(\nicefrac1p\, |a|^p+\nicefrac1q\, |b|^q \right)
\end{equation}
for all $s\ge 0$. The purpose of this paper is to attack the following conjecture: 

\smallskip

{\it{Let $p,q>1$ with $1/p+1/q=1$. Does}}
\begin{equation}\label{sigual}
\mu_s(ab^*)=\mu_s\left(\nicefrac1p\, |a|^p+\nicefrac1q\, |b|^q\right)
\end{equation}
{\it{for all $s> 0$ imply $|a|^p=|b|^q$?}}

\medskip

\begin{rem}
If the algebra $\mathcal A$ is atomic, we have already answered in the affirmative the conjecture in \cite[Theorem 2.12]{hycomp}. There, we used the existence of eigenvectors for each non-trivial eigenvalue. In this paper we will be dealing with the continuous case (that contains the previous one, see Section \ref{difusemi}), using continuous techniques.
\end{rem}

\subsection{Extension to unbounded operators}

We extend the inequality and the conjecture to unbounded operators. 

\begin{teo}
Let $a,b\in \widetilde{\a}$, then for each $s>0$
\begin{equation}\label{ysf}
\mu_s(ab^*)\le \mu_s\left(\nicefrac1p\, |a|^p+\nicefrac1q\, |b|^q \right).
\end{equation}
\end{teo}
\begin{proof}
Let $a=u|a|$, $b=\nu|b|$ be the polar decompositions of $a,b$. Approximating $|a|,|b|$ in measure from below with bounded operators $x_n,y_n\ge 0$, we have for each $s>0$
$$
\mu_s(x_ny_n)\le \mu_s\left(\nicefrac1p\, x_n^p+\nicefrac1q\, y_n^q\right)\le \mu_s\left(\nicefrac1p\, |a|^p+\nicefrac1q\, |b|^q\right)
$$
by (\ref{fm}) applied to the pair $x_n,y_n$ and Lemma \ref{singv}.1. Since $x_n\le |a|$, $y_n\le |b|$, it is easy to check that $|x_ny_n|\le | |a||b||$; since $|ab^*|=\nu||a||b||\nu^*$, then $\mu_s(x_ny_n)\le \mu_s(ab^*)$. Since  $x_ny_n$ converges in measure to $|a||b|$, then  by \cite[Lemma 3.4]{fk}, $\lim_n\mu_s(x_nyn)=\mu_s(ab^*)$ for each $s>0$, proving the claim.
\end{proof}

\subsection{Some restrictions and simplifications}

To make sense out of the conjecture (\ref{sigual}), we should ask for a complete description of an operator in terms of its $s$-numbers. We therefore think that it is natural to the confine the conjecture to the ideal $\ka$ of $\tau$-compact operators (Remark \ref{compa}). 

In fact, it is known that  for $ x\in \ka^+$,
$$
\sigma(x)=\mbox{clos}\{\mu_s(x):s>0\}
$$
(see \cite[Theorem 4.10]{stroh}). On the other hand, if $e,f$ are disjoint and infinite projections ($\tau(e)=\tau(f)=\infty$), taking $x=e+\frac12 f$ shows that $\sigma(x)=\{1/2,1\}$ while $\mu_s(x)=1$ for all $s>0$, therefore it is hopeless to recover $x$ from the data in $\mu_s(x)$.

\medskip

Exchanging $a$ with $b$, we can always assume that $1<p\le 2$. Since $\lambda_s=\mu_s(|ab^*|)=\mu_s(|a||b|)$, we can safely assume that $a,b\ge 0$. Moreover, we can assume (see Section \ref{difusemi}) that $\mathcal A$ is diffuse and there exist complete flags $e_t,q_t\in\pa$ ($t\ge 0$, $\tau(e_t)=\tau(q_t)=t$) such that
\begin{equation}\label{intflag}
|ab|=\int_0^{\infty} \lambda_s de(s)  \quad \textit{ and } \quad\frac1p a^p+ \frac1q b^q=\int_0^{\infty} \lambda_s dq(s),
\end{equation}
since $a,b\in\ka^+$ and $\tau$-compact operators form a (closed in measure) ideal of $\widetilde{\a}$. 

\medskip

Our arguments will be based on continuous majorization. We are therefore interested in  those operators that are locally integrable. More precisely, let $1\le p <\infty$, let $x\in \a$ and assume that there exists $\delta>0$ such that
$$
\int_0^{\delta} \mu_s(x)^p ds<\infty
$$
(hence the integral is finite for all finite $\delta>0$). We will denote the set containing all these operators  by $\loclp\subset \widetilde{\a}$. Note that in particular, all  bounded operators $a\in \a$ are of this class. Moreover,
$$
\int_0^\delta\mu_s(x)^pds\ge \mu_{\delta}(x)^{p-1}\int_0^\delta \mu_s(x)ds
$$
shows that $\loclp\subset\lo1 $ for each $p\ge 1$.

\begin{lem}\label{loco}
Let $a\in\widetilde{\a}$ and $p\ge 1$. Then $a\in \loclp$ if and only if $a\in \elep+\a$, and in that case the decomposition can be taken as follows for some $r>0$.
\begin{equation}\label{l1}
a=ap^a(r,+\infty)+ ap^a[0,r].
\end{equation}
\end{lem}
\begin{proof}
By polar decomposition, it suffices to consider $a\ge 0$. Note that $ap^a[0,r]\le r\in\a$ and since $a\in\a$, eventually $\tau(p^ a(r,+\infty))<\infty$ for some $r>0$. Likewise, for $p>1$,
$$
a^p=a^p p^a(r,+\infty)+ a^p p^a[0,r].
$$
These expressions imply the following (see \cite[Proposition 1.2]{hiaimayo}):
$$
a^p\in \lo1 \Leftrightarrow a^p\in \l1+\a \Leftrightarrow \mbox{ there exists }r>0 \mbox{ such that }a^p p^a(r,+\infty)\in \l1.
$$
Note that then $a p^a(r,+\infty)\in \elep$ for the same $r$, therefore $a\in \elep+\a$ by (\ref{l1}). On the other hand, if $a=l+m\in \elep+\a$, then taking $f(x)=x^p$ which is continuous, convex and increasing in $[0,+\infty)$,
$$
\int_0^t \mu_s(a)^pds=\int_0^t\mu_s(l+m)^pds\le \int_0^t (\mu_s(l)+\mu_s(m))^p ds
$$
by \cite[Lemma 4.4.iii]{fk}. Therefore for any $t>0$
\begin{eqnarray}
\left(\int_0^t \mu_s(a)^pds\right)^{1/p}&\le& \left( \int_0^t (\mu_s(l)+\mu_s(m))^p ds\right)^{1/p}\nonumber\\
&\le& \left(\int_0^t \mu_s(l)^pds\right)^{1/p}+ \left(\int_0^t \mu_s(m)^pds\right)^{1/p}\nonumber\\
&\le &\left(\int_0^t \mu_s(l)^pds\right)^{1/p} +\|m\|t^{1/p}<\infty\nonumber
\end{eqnarray}
by the classical Mikowkski inequality, therefore $a\in \loclp$. Take $r=\mu_t(a)$, and note that for all $s>0$,
$$
\mu_s(ap^a(r,+\infty))=\left\{\begin{array}{ll} \mu_s(a) & 0<s<t\\
0 & s\ge t
\end{array}\right.,
$$
therefore  (\ref{l1}) gives the stated decomposition.
\end{proof}

\section{Main results}

We start by examining the ranges of $a,b$. Throughout, $p,q$ are positive with  $1/p+1/q=1$.

\begin{prop}\label{pnoes2conti}
Let $0\le a,b\in\ka$ with $ab\in \loc2$. If $p\ne 2$  and
$$
\mu_s(ab)=\mu_s\left(\frac1p a^p +\frac1q b^q\right)\quad \textit{ for all }s>0
$$
then $\overline{\ran(a)}=\overline{\ran(b)}$.
\end{prop}
\begin{proof}
Exchanging $a,b$ it will suffice to consider $1<p<2$. Let $p_b$ be the projection onto the closure of the range of $b$. Let $\be=b+\varepsilon(1-p_b)$, then $\be^q=b^q+\varepsilon^q (1-p_b)$ and $\be^2=b^2+\varepsilon^2(1-p_b)$. Fix $t>0$, let $\{e_s\}$ be a complete flag for $|ba|$, then denoting $\lambda_s=\mu_s(ab)$ we have 
$$
\int_0^t\lambda_s^2 de(s)+\varepsilon^2 e_ta(1-p_b)ae_t=e_t|ba|^2e_t+\varepsilon^2 e_ta(1-p_b)ae_t=e_t|\be a|^2 e_t.
$$
Taking the trace, it follows that
$$
\int_0^t\lambda_s^2 ds+\varepsilon^2 \tau(e_ta(1-p_b)ae_t)=\tau(e_t|\be a|^2)\le \int_0^t\mu_s(|\be a|)^2
$$
by Remark \ref{intemu}. On the other hand, by (\ref{ysf}) applied to $a,b_{\varepsilon}$,
\begin{eqnarray}
\mu_s(|\be a|)&\le & \mu_s\left(\frac1p a^p+\frac1q \be^q \right)=\mu_s\left(\frac1p a^p +\frac1q b^q+ \frac1q \varepsilon^q(1-p_b) \right)\nonumber\\
&\le & \mu_s\left(\frac1p a^p+\frac1q b^q \right)+ \frac1q\varepsilon^q= \mu_s(ab)+\frac1q\varepsilon^q=\lambda_s+ \frac1q\varepsilon^q\nonumber.
\end{eqnarray}
Note that in particular, all the integrals computed up to now are finite by the hypothesis on $ab$, and
$$
\int_0^t\lambda_s^2 ds+\varepsilon^2 \tau(e_ta(1-p_b)ae_t)\le \int_0^t\lambda_s^2 ds+\frac{1}{q^2}t\varepsilon^{2q}+\frac2q \varepsilon^q\int_0^t\lambda_s ds.
$$
Cancelling $\int_0^t\lambda_s^2ds$ and dividing by $\varepsilon^2$, noting that $q>2$ and letting $\varepsilon\to 0$ gives us that $\tau(e_ta(1-p_b)ae_t)=0$. Since the trace is faithful, we conclude that $(1-p_b)ae_t=0$ or equivalently, $ae_t=p_bae_t$ for all $t>0$. Then 
$$
p_b a|ba|=p_ba\int_0^{\infty}\lambda_t de(t)=\int_0^{\infty}\lambda_t p_ba de(t)=\int_0^{\infty}\lambda_t a de(t)=a|ba|, 
$$
that is $a(\ran|ba|)\subset \overline{\ran(b)}$.

Now if $\xi\in \mathcal H$, then $a|ba|\xi\in\overline{\ran(b)}$, therefore $a^2|ba|\xi=a(a|ba|\xi)\in a\overline{\ran(b)}\subset \overline{\ran(ab)}=\overline{\ran|ba|}$, and $a^3|ba|\xi=a(a^2|ba|\xi)\in a \overline{\ran|ba|}\subset \overline{\ran(b)}$.  Iterating this argument, we arrive to the conclusion that $a^{2n+1}(\ran|ba|)\subset\overline{\ran(b)}$ for all $n\in\mathbb N_0$. Using an approximation of $f=\chi_{\sigma(a)}$ by odd functions, we conclude that $p_a(\ran|ba|)=f(a)(\ran|ba|)\subset\overline{\ran(b)}$ where $p_a$ is the projection onto the closure of the range of $a$. Therefore $|ba|^2\xi=ab^2a\xi=p_aab^2a\xi=p_a|ba|^2\xi\subset \overline{\ran(b)}$,  which gives $\overline{\ran|ba|}=\overline{\ran(|ba|^2)}\subset \overline{\ran(b)}$. But then 
$$
a\overline{\ran(b)}=\overline{\ran(ab)}=\overline{\ran|ba|}\subset \overline{\ran(b)}
$$
which proves that the range of $b$ is invariant for $a$; since $a\ge 0$ the same is true for the kernel of $b$. Therefore we can write  $a=a_b+a_\perp$, with $a_b=p_bap_b\ge 0$ and $a_\perp=(1-p_b)a(1-p_b)\ge 0$. Note that $ba^2b=ba_b^2b$ and $a^p=a_b^p+a_{\perp}^p$, thus for all $s>0$,
\begin{eqnarray}
\mu_s\left(\frac1p a_b^p+\frac1q b^q \right)&\le & \mu_s\left(\frac1p a_b^p+\frac1q b^q +\frac1p a_{\perp}^p\right)=\mu_s\left(\frac1p a^p+\frac1q b^q \right)=\mu_s(ab)\nonumber\\
&=& \mu_s(a_bb)\le \mu_s\left(\frac1p a_b^p+\frac1q b^q \right)\nonumber
\end{eqnarray}
by the hypothesis and (\ref{ysf}) applied to $a_b,b$. This proves that for all $s>0$
$$
\mu_s\left(\frac1p a_b^p+\frac1q b^q \right)=\mu_s\left(\frac1p a_b^p+\frac1q b^q+ \frac1p a_{\perp}^p \right)
$$
which (by Lemma \ref{singv}.5) is only possible if $a_{\perp}=0$, proving the assertion of the proposition.
\end{proof}

\medskip

The following will be used twice throughout the proof of the main theorem, therefore we preferred to state it as a separate lemma:

\begin{lem}\label{nojodan}
Let $0\le x\in \lo1$  and $p\in\a$ a projection with finite trace. Then
$$
\tau(px)=\int_0^{\tau(p)}\mu_s(x)ds
$$
implies $xp=px$.
\end{lem}
\begin{proof}
Since $p$ is a projection and $x\ge 0$,
$$
(pxp)^2=pxpxp\le px^2p=|xp|^2.
$$
Since the square root is operator monotone, $pxp\le |xp|$.  Take the trace and invoke items 4 and 5 of Lemma \ref{singv}, then
$$
\tau(|xp|)=\int_0^{\infty} \mu_s(xp)ds=\int_0^{\tau(p)}\mu_s(xp)ds\le \int_0^{\tau(p)}\mu_s(x)ds,
$$
thus by the hypothesis $pxp$ and $|xp|$ have equal (and finite) trace. Since the trace is faithful this is only possible if $pxp=|xp|$, or equivalently if $pxpxp=px^2p$. This implies that 
$$
\langle px,xp\rangle_2=\tau(pxpx)=\tau(px^2p)=\|px\|_2^ 2=\|px\|_2\|xp\|_2,
$$
which by the case of equality in Cauchy-Schwarz inequality implies $xp=px$.
\end{proof}
\medskip

We will also need the following classical result on operator ranges for  \cite[Theorem 2]{doug}.

\begin{rem} \textbf{(Douglas' Lemma).}\label{dlema}
Let $x,y\in \widetilde{\a}$. If $xx^*\le \lambda yy^*$ for some $\lambda\ge 0$, there exists a contraction $c$ such that $x=yc$, therefore  $\ran(x)\subset \ran(y)$.
%( there exists $c\in \a$ such that $yc=x$.)
%In that case, there exists a unique such $c$ such $\ker(c)=\ker(x)$, $\ran(c)\subset\overline{\ran(y^*)}$ and  $\|c\|^2=\inf\{\lambda:xx^*\le \lambda yy^*\}$, that is  the bound %$$xx^*\le \|c\|^2 yy^*$$ is optimal.
\end{rem}

With this tools at hand, we are now able to prove the main theorem.

\begin{teo}\label{sii}
Let $0\le a,b\in \ka$ with  $ab\in \loc2$. If
$$
\mu_s(ab)=\mu_s\left(\frac1p a^p +\frac1q b^q\right)\quad \textit{ for all }s>0,
$$
then $a^p=b^q$. When $p=q=2$ it suffices to assume $ab\in \lo1$.
\end{teo}
\begin{proof}
Exchanging $a,b$ it will suffice to consider $1<p\le 2$. Denoting $\lambda_s=\mu_s(ab)$ we write $ba^2b=|ab|^2=\int_0^{\infty}\lambda_s^2 de(s)$ with a complete flag $\{e_s\}_{s\ge 0}\subset\pa$. For $I=[s,t]\subset [0,+\infty)$ denote $e_I=e_s-e_t=e(s,t)$, then $e_t=e_{[0,t]}$. Since $\lambda_s$ is non-increasing,
\begin{equation}\label{ei}
ba^2b= ba(ba)^*=|ab|^2=\int_0^\infty \lambda_s^2de(s)\ge \int_I \lambda_s^2 de(s)\ge \lambda_t^2 \int_s^t de(s)=\lambda_t^2 e_I,
\end{equation}
and the previous lemma ensures that $\ran(e_I)\subset\ran(|ab|)=\ran(ba)\subset \ran(b)$ for each interval $I=[s,t]$.  Moreover,  if $e=\bigvee_{s\ge 0}e_s$ is the join of the increasing projections,  clearly $e =p_{|ab|}$, the projection onto the closure of the range of $|ab|$. 

We now treat three cases separately.

\medskip

\textbf{Case $4/3\le p<2$.} By Proposition \ref{pnoes2conti} we can consider $\mathcal H=\overline{\ran(a)}=\overline{\ran(b)}$, and the semi-finite von Neumann subalgebra ${\mathcal M}\subset \a$ generated by the (finitely supported) spectral projections of $a,b,ab$. We give $\mathcal M$ the inherited trace $\tau$ and identity $1=1_{\mathcal M}=P_{\mathcal H}=p_b$. All the operators involved $a,b,ab,a^p,b^q$  are in $\widetilde{\mathcal M}$, and $\mathcal M$ can be faithfully represented in this $\bh$. Then we can safely assume that $b$ is injective, $e,e_I\in \mathcal M$ for each interval $I$, and  $\mathcal M\subset \mathcal H$ is a common core for all $x\in\widetilde{\mathcal M}$.

We remark that in what follows, \textit{we will only use that $b$ is injective}, or equivalently, that the range of $b$ is dense.

Again for each interval $I$, let $\mathcal H_I$ be the closure of $b^{-1}\ran(e_I)\subset \mathcal H$.  Let $f_I=P_{\mathcal H_I}$ and $f=\bigvee_{s\ge 0} f_s$ the closed join of all projections, where $f_s=f_{[0,s]}$. We divide the proof in several smaller claims.

\smallskip

Claim: $f=p_a$, the projection onto the closure of the range of $a$. Let $\eta\in\ran(f_s)$; then $\eta=\lim_n\eta_n$ with $b\eta_n=e_s\xi_n$; since $\ran(e_s)\subset\ran(|ab|)=\ran(ba)$, it must be $b\eta_n=ba\psi_n$ for some $\psi_n\in\mathcal H$, and by the injectivity of $b$, we obtain $\eta_n\in \ran(a)$, therefore $\eta\in\overline{\ran(a)}$. This proves that $f\le p_a$. On the other hand, if $\eta\in \ran(a)$, then $b\eta\in\ran(ba)=\ran|ab|\subset \ran(e)$, therefore $b\eta=\lim \xi_n$ with $\xi_n\in \ran(e_{s_n})$ for some $s_n>0$. Therefore $\xi_n=b\eta_n$ with $\eta_n\in\ran(f_{s_n})\subset \ran(f)$ for each $n$. Now
$$
|\langle \eta_n -\eta_m,b\xi\rangle| =|\langle b\eta_n -b\eta_m,\xi\rangle|=|\langle \xi_n-\xi_m,\xi\rangle|\le \|\xi_n-\xi_m\|\|\xi\|
$$
and since the range of $b$ is dense, $\{\eta_n\}_n$ is a weak Cauchy sequence in $\ran(f)$ which, being closed and linear, it is weakly closed. Therefore $\eta_n$ converges weakly to some $\eta_0\in \ran(f)$. But for each $\xi\in\mathcal H$, 
$$
\langle b\eta, \xi\rangle =\lim_n \langle b\eta_n,\xi\rangle=\lim_n\langle \eta_n,b\xi\rangle =\langle \eta_0,b\xi\rangle =\langle b\eta_0,\xi\rangle,
$$
which implies that $b\eta=b\eta_0$, and by the injectivity of $b$, we obtain $\eta=\eta_0\in \ran(f)$, thus $p_a\le f$. This proves that $f=p_a$.

\smallskip

Claim: there exists a closed operator $c_I="b^{-1}e_Ib^{-1}"$ defined on $\ran(b)$ such that $0\le c_I\le  \frac{1}{\lambda_t^2}a^2$. Since $b$ is injective and $\ran(e_I)\subset \ran(b)$, for each $\xi\in \mathcal H$ there exists a unique $\eta\in\mathcal H_I$ such that $b\eta=e_I\xi$. Define $c_I$ in $\ran(b)$ as follows: $c_Ib\xi=\eta$. We now compute
$$
\langle c_I b\xi_1,b\xi_2 \rangle = \langle \eta_1,b\xi_2\rangle =\langle b\eta_1,\xi_2\rangle =\langle e_I\xi_1, \xi_2\rangle
$$
which shows that $c_I$ is a symmetric operator on $\ran(b)$. Moreover, since $e_I\le \frac{1}{\lambda_t^2}ba^2b$ (recall $I=[s,t]$ and equation (\ref{ei})), it follows that 
$$
\langle c_I b\xi,b\xi \rangle = \langle \eta,b\xi\rangle =\langle e_I\xi, \xi\rangle\le \frac{1}{\lambda_t^2}\langle ba^2b\xi,\xi\rangle = \frac{1}{\lambda_t^2}\langle a^2 b\xi,b\xi\rangle,
$$
therefore $c_I$ has a self-adjoint extension (c.f. \cite[Theorem 5.1.13]{ped}, that we still denote $c_I$), and $0\le c_I\le  \frac{1}{\lambda_t^2}a^2$.

Claim: $c_I\in\widetilde{\mathcal M}$. Let $u\in \mathcal M'$, let $\xi\in\mathcal H$. Then there exists unique  $\eta,\psi\in\mathcal H$ such that $e_I\xi=b\psi$ and $e_I(u\xi)=b\eta$. Now $e_Iu\xi=ue_I\xi$ since $u\in\mathcal M'$, therefore 
$$
b\eta=e_I u\xi=u e_I\xi=u b\psi=b u\psi,
$$
and since $b$ is injective, $u\psi=\eta$. We now compute
$$
c_I u(b\xi)=c_I b(u\xi)=\eta=u\psi=u(c_Ib\xi)=uc_I(b\xi),
$$
which shows that $c_Iu=uc_I$, proving that $c_I\in\widetilde{\mathcal M}$.

\smallskip

Claim: $bc_Ib=e_I$, $e_Ibf_I=bf_I$ and $f_Ib^2f_Ic_I=f_I$ for each $I$. From the very definition, $bc_Ib=e_I$. On the other hand note that if $\eta\in f_I$ then $\eta=\lim_n\eta_n$ with $b\eta_n\in \ran(e_I)$ and for any $\xi\in\mathcal H$
$$
|\langle b\eta_n -b\eta,\xi\rangle|=|\langle \eta_n-\eta,b\xi|\le\|\eta_n-\eta\|\|b\xi\|
$$
therefore $b\eta_n\in \ran(e_I)$ converges weakly to $b\eta$, therefore $b\eta\in \ran(e_I)$ and we obtain $e_Ibf_I=bf_I$. Taking adjoints, $f_Ib=f_Ibe_I$, hence for each $\xi\in\mathcal H$, 
$$
f_Ib^2f_Ic_I(b\xi)=f_Ib^2f_I\eta=f_Ib^2\eta=(f_Ib)(b\eta)=(f_Ib)(e_I\xi)=f_I(b\xi), 
$$
which proves that $f_Ib^2f_Ic_I=f_I$ for any $I$.

\smallskip

Claim: $f_Ic_J=c_Jf_I$ for any $I,J$. Since when $e_I\xi=b\eta$, then $\eta\in\ran(f_I)$, clearly $f_Ic_I =c_I=c_If_I$. Moreover it is not hard to see that $f_Ic_J=c_{I\cap J}$ for any pair of intervals $I,J$ by the injectivity of $b$, therefore $f_Ic_J=c_Jf_I$. 

\smallskip

Claim: $f_I\sim e_I$. Inspection of the ranges shows that (again by the injectivity of $b$)
$$
\ran(e_I)=\overline{\ran(bc_I)}\textit{ and }\ran(f_I)=\overline{\ran(c_Ib)},
$$
and since $c_Ib=(bc_I)^*$, it follows that $e_I$ is von Neumann equivalent to $f_I$, which implies that $f,f_I\in\mathcal M$ and moreover $\tau(f_I)=\tau(e_I)$ for each $I$. 

\smallskip

Summing up our findings:  for any interval $I=[s,t]\subset [0,+\infty)$, we have $e,e_I,f,f_I\in \mathcal P(\mathcal M)$ with $f_I\simeq e_I$, $\tau(f_I)=	\tau(e_I)=t-s$, $e_Ibf_I=bf_I$, $f_Ib=f_Ibe_I$. Moreover $c_I\in\widetilde{\mathcal M}$,  $bc_Ib=e_I$,
\begin{equation}\label{findings}
f_Ic_J=c_{I\cap J}=c_Jf_I,\quad  f_Ib^2f_I c_I=f_I=c_If_Ib^2f_I,\quad  a^2\ge \lambda_t^2 c_I\, \mbox{  and  } f_Ia^2f_I\ge \lambda_t^2 c_I.
\end{equation}

Claim: $a^pf_s=f_sa^p$ for all $s>0$. Let $\pi=\{I_i\}_{i=1\cdots n}$ with $I_i=[s_i,s_{i+1}]$ be a partition of $[0,+\infty)$, and denote $e_i=e_{I_i}$ and likewise with $f_i,c_i$. We have
$$
ba^2b\ge \sum_i\int_{I_i}\lambda_s^2 de(s)\ge \sum_i  \lambda_{s_{i+1}}^2 e_i= \sum_i  \lambda_{s_{i+1}}^2bc_ib
$$
which implies $a^2\ge \sum_i  \lambda_{s_{i+1}}^2c_i$ since $b$ is injective with dense range. Now refining the partition
\begin{eqnarray}
\langle a^2b\xi,b\xi\rangle & = & \langle ba^2 b\xi,\xi\rangle = \langle \int_0^{+\infty}\lambda_s^2 de(s)\xi,\xi\rangle=\lim\limits_{|\pi|\to 0} \langle \sum_i\lambda_{s_{i+1}}^2 e_i \xi ,\xi\rangle\nonumber\\
&=& \lim\limits_{|\pi|\to 0} \langle \sum_i\lambda_{s_{i+1}}^2c_i \, b\xi ,b\xi\rangle  \nonumber
\end{eqnarray}
for any $\xi \in\mathcal H$. Since the range of $b$ is dense and the operators involved are positive, we conclude that $\lim\limits_{|\pi|\to 0} \sum_i\lambda_{s_{i+1}}^2 c_i =a^2$ in the strong operator topology. Since $f_sc_i=c_if_s(=c_{[0,s]\cap I_i})$, we conclude that $f_sa^2=a^2f_s$ for all $s\ge 0$, which implies that $a^pf_i=f_ ia^p$ for all $i$.

\smallskip

We now take the $p/2$-th root in (\ref{findings}), which is a monotone operator function since $1<p<2$. Thus
\begin{equation}
\label{ac}
\lambda_{s_{i+1}}^p c_i^{p/2}\le a^p\; \text{ and  }\; \lambda_{s_{i+1}}^p c_i^{p/2}\le f_ia^pf_i=a^pf_i.
\end{equation}
Since $ 4/3\le p<2$, this implies that $2<q\le 4$. Then $t\mapsto t^{q/2}$ is operator convex \cite[Theorem 2.4]{ped} and $(f_ib^2f_i)^{q/2}\le f_i b^qf_i$. By Young's inequality in the commutative algebra generated by $c_i,f_i$ \cite[Lemma 2.2]{efz}
\begin{eqnarray}\label{younging}
\lambda_{s_{i+1}} f_i&=&\lambda_{s_{i+1}} f_i^{1/2}=\lambda_{s_{i+1}} c_i^{1/2}(f_ib^2f_i)^{1/2} \le \frac1p\lambda_{s_{i+1}}^p c_i^{p/2}+ \frac1q (f_ib^2f_i)^{q/2} \\
&\le &\frac1p f_ia^pf_i +\frac1q f_ib^qf_i =f_i\left(\frac1p a^p+\frac1q b^q\right)f_i=f_iDf_i,\nonumber
\end{eqnarray}
where $D=\frac1p a^p+\frac1q b^q$ for short. 

Claim: $f_tD=Df_t$ for all $t>0$. Assume that $\pi$ is a partition of $[0,t]$. Summing over $i$, we obtain $\sum_i \lambda_{s_{i+1}} f_i\le \sum_i f_iDf_i$, and taking traces 
$$
\sum\limits_i\lambda_{s_{i+1}} (s_{i+1}-s_i) \le \sum_i\tau(f_iD)=\tau(Df_t)\le \int_0^t \mu_s(D)ds=\int_0^t\lambda_s ds<\infty
$$
by (\ref{intflag}), Remark \ref{intemu} and the assumption on $a,b$ (recall $\loc2\subset \lo1$). Refining the partition $\pi$, it follows that $\int_0^t\lambda_s ds = \tau(Df_t)$. Since $\lambda_s=\mu_s(D)$ and $\tau(f_t)=t$, Lemma \ref{nojodan} implies that  $f_tD=Df_t$.

\smallskip

Claim: $D=\int_0^\infty \lambda_s df(s)$. Since $t$ was arbitrary, $f_iD=Df_i$ also holds. Returning to the previous inequality (\ref{younging}) we now sum over $i$ to obtain
$$
\sum_i\lambda_{s_{i+1}}f_i\le \sum_i f_iDf_i =\sum_i f_i D= f_t D=f_tD.
$$
Let $\overline{D}=\int_0^\infty \lambda_s df(s)$, then $(1-f)\overline{D}=0$ and $\overline{D}f_t=\int_0^t\lambda_s df(s)$. Refining the partition $\pi$ of $[0,t]$ we obtain $\overline{D}f_t\le f_t D$, and since $\tau(\overline{D}f_t)=\tau(f_tD)=\int_0^t\lambda_s ds$, it must be $\overline{D}f_t=Df_t=f_tD$ for each $t>0$. Recall $f=\bigvee_t f_t$ is the union of the projections $f_t$, then clearly $\overline{D}=Df=fD$; on the other hand
$$
\mu_s(\overline{D})=\lambda_s=\mu_s(D)=\mu_s(Df+(1-f)D)=\mu_s(\overline{D}+(1-f)D)
$$
and by Lemma \ref{singv}.3 it is only possible if $(1-f)D=0$, or equivalently 
$D=\overline{D}$.

\smallskip

Claim: $a$ commutes with $b$, $b$ commutes with all $f_t$ and $b^qf=b^q$. Since $a^p$ commutes with all $f_s$, then $a^p$ commutes with 
$$
D=\int_0^{\infty}\lambda_sdf(s)=\frac1p a^p+\frac1q b^q.
$$
Then $a^p$ commutes with $b^q$ or equivalently, $a$ commutes with $b$. Note that since $f=p_a$, then 
$$
\frac1p a^p+\frac1q b^q=D=Df=\frac1p a^p+\frac1q b^qf,
$$
therefore $b^qf=b^q$. Since $a^p$ commutes with all $f_t$ and $\frac1q b^q=D-\frac1p a^p$, then $b^q$ commutes with all $f_t$.

\smallskip

Claim: $f_t=e_t$ for all $t$. Recall that for all $t$, $bf_t=e_tbf_t$, therefore $f_tbe_t=f_tb$. Since $f_t$ commutes with $b$, $bf_te_t=bf_t$; since $b$ is injective, $f_te_t=f_t$. Therefore $f_t=f_te_t=e_tf_te_t\le e_t$, and since $\tau(f_t)=\tau(e_t)=t$, it must be $f_t=e_t$ for all $t$.

\smallskip

Finally,
$$
|ab|=ab=\int_0^{\infty}\lambda_s df(s)=\frac1p a^p+\frac1q b^q.
$$
Let $a_t=af_t,b_t=bf_t$, then  $\frac1p\mu_s(a_t)^p\le \mu_s(D_t)=\lambda_s\in \lo1$ for $0<s<t$ and likewise with $b$. This means that $a_t^p,b_t^q\in \lo1$ and
$$
|a_tb_t|=a_tb_t=abf_t=\int_0^t \lambda_s df(s)=\frac1p a_t^p+\frac1q b_t^q,
$$
and  all the operators involved have \textit{finite} trace. Farenick and Manjegani proved that in that case (see \cite[Theorem 3.1]{fm} or \cite[Theorem 2.1]{m}), it must be $a^pf_t=a_t^p=b_t^q=b^qf_t$. We give here an alternative argument: taking traces
$$
\frac1p\|a_t\|_p^p+\frac1q\|b_t\|_q^q=\tau\left(\frac1p a_t^p+\frac1q b_t^q\right)=\tau(|a_tb_t|)=\|a_tb_t\|_1\le \|a_t\|_p\|b_t\|_q\le \frac1p+ \|a_t\|_p^p+\frac1q\|b_t\|_q^q
$$
by the operator H\"older inequality (applied to $\|a_tb_t\|_1$) and Young's numeric inequality (applied to $\|a_t\|_p,\|b_t\|_p$). This implies $\|a_tb_t\|_1=\|a_t\|_p\|b_t\|_q$, and this is only possible if $a_t^p=b_t^q$ \cite{dixmier,holg}. Since this holds for all $t>0$,  $a^p=a^pf=b^qf=b^q$ as we claimed.

\medskip

\textbf{Case $1<p<4/3$.} This implies that $q>4$, but since the ranges of $a$ and $b$ still match by Proposition \ref{pnoes2conti}, we can assume that $b$ is injective with dense range, and the computation goes through the same lines, modifying the step regarding   the commutative operator Young inequality (\ref{younging}) according to \cite[Theorem 2]{ando} or \cite[Proposition 2.3]{efz}. 

\medskip

\textbf{Case $p=q=2$}. First note that 
$$
\frac12 \mu_s(a)^2\le \mu_s\left(\frac12 a^2+\frac12 b^2\right)=\mu_s(ab)\in \loc2\subset \lo1,
$$
therefore $\mu_s(a)\in \loc2$ and likewise with $b$. Proposition \ref{pnoes2conti} is of no use here, therefore it suffices to assume $\mu_s(ab)\in \lo1$.

Let $\tilde{a}=(p_ba^2p_b)^{1/2}$. Then $\ran(\tilde{a})\subset\ran(b)$ and $b\tilde{a}^2b=ba^2b$, therefore $|\tilde{a}b|=|ab|$. Hence
\begin{eqnarray}
\mu_s\left(\frac12 \tilde{a}^2+\frac12 b^2\right) &=& \mu_s\left(p_b(\frac12 a^2+\frac12 b^2)p_b\right)\le \mu_s\left(\frac12 a^2+\frac12 b^2\right)=\mu_s(ab)=\mu_s(\tilde{a}b)\nonumber\\
&\le& \mu_s\left(\frac12 \tilde{a}^2+\frac12 b^2\right)\nonumber
\end{eqnarray}
by (\ref{ysf}) applied to the pair $\tilde{a},b$. Therefore, for all $s\ge 0$,
$$
\mu_s(\tilde{a}b)=\mu_s\left(\frac12 \tilde{a}^2+\frac12 b^2\right).
$$
Since $\ran(\tilde{a})\subset \overline{\ran(b)}$, we can assume that $b$ is injective, and argumenting as in the previous cases, arrive to $\tilde{a}^2=b^2$, that is $p_ba^2p_b=b^2$. In particular $\mu_s(b)^2\le \mu_s(a)^2$ for all $s> 0$. Reversing the argument, we also get $p_ab^2p_a=a^2$, therefore $\mu_s(a)=\mu_s(b)$ for all $s> 0$.

Let $\{b_s\}_{s\ge 0}$ be a complete flag for $b=\int_0^{\infty}\mu_s(b) db(s)$ with $\tau(b_t)=t$. Then for all $t\ge 0$,  $b_t$ commutes with $b$, we have  $b_tb=\int_0^t\mu_s(b)db(s)$ and since $b_t\le p_b$, $b_tp_b=b_t$. Therefore from $p_ba^2p_b=b^2$ we obtain $b_ta^2b_t=b_tb^2$, which implies that 
$$
\tau(b_ta^2)=\tau(b_tb^2)=\int_0^t\mu_s(b)^2ds=\int_0^t\mu_s(a)^2ds<\infty.
$$
By Lemma \ref{nojodan}, this is only possible if $a$ commutes with $b_t$. Therefore, $a$ commutes with $b$, then from $p_a b^2p_a=a^2$ we have $bp_a=p_ab=a$. But
$$
\mu_s(a)=\mu_s(b)=\mu_s(bp_a+ (1-p_a)b)=\mu_s(a+(1-p_a)b)
$$
implies (Lemma \ref{singv}.6) $b=p_ab=a$.
\end{proof}

\bigskip

\begin{rem}\label{rangos}
As the proof goes, it suffices to consider $ab\in \lo1$ if either
$$
\ran(a)\subset\overline{\ran(b)}\mbox{ or }\ran(b)\subset\overline{\ran(a)}.
$$
\end{rem}

\begin{coro}\label{acot}
Let $0\le a,b\in \a\cap \ka$  and assume
$$
\mu_s(ab)=\mu_s\left(\frac1p a^p +\frac1q b^q\right)\quad \textit{ for all }s>0.
$$
Then $a^p=b^q$.
\end{coro}

\subsection{Symmetric norms}

We close the paper putting this result in context with the theory of symmetric norms on $\widetilde{\a}$, see for instance \cite{dodds} and the references therein.

\medskip

We say that a symmetric norm $\|\cdot\|_E$ is \textit{strictly increasing} if $x,y\in E\subset\widetilde{\a}$, $\mu_s(x)\le \mu_s(y)$ for all $s>0$ and $\|x\|_E=\|y\|_E$ implies $\mu_s(x)=\mu_s(y)$ for all $s>0$. All ${\mathscr L}^p$-norms are strictly increasing for $1\le p<\infty$, while the uniform norm or the Ky-Fan norms $\|x\|_{(t)}=\int_0^t \mu_s(x)ds$ are not.

\begin{teo}\label{elteo}
Let $a,b\in \ka\cap \loc2$. If $p>1$ and $1/p+1/q=1$, then the following are equivalent:
\begin{enumerate}
\item $|a|^p=|b|^q$.
\item $z|ab^*|z^* =\frac1p |a|^p +\frac1q |b|^q$ for some contraction $z\in \a$ 
\item $\|z|ab^*|w\|_{E} =\|\frac1p |a|^p +\frac1q |b|^q\|_{E}$ for a pair of contractions $z,w\in\a$ and $\|\cdot\|_{E}$ a \textbf{strictly increasing} symmetric norm. 
\item $\mu_s(ab^*)=\mu_s\left(\frac1p |a|^p+\frac1q |b|^q\right)$ for all $s>0$.
\end{enumerate}
\end{teo}
\begin{proof}
The proof is much like as in \cite[Theorem 2.13]{hycomp}, therefore it is omitted.
\end{proof}

As in Theorem \ref{sii}, Remark \ref{rangos} or Corollary \ref{acot}, the hypothesis $ab\in \loc2$ is unnecessary when $ab$ is bounded, and can be relaxed to $ab\in \lo1$ if $p=q=2$ or if there is an inclusion of ranges.

\bigskip

\noindent
Gabriel Larotonda\\
Instituto ''Alberto P. Calder\'on'' (CONICET)\\
and Universidad de Buenos Aires, Argentina  \\
e-mail: glaroton@dm.uba.ar

\end{document}